\documentclass[12pt]{article}
\usepackage{amsmath,amsthm,amssymb,fullpage}
\usepackage{color}                    

\newcommand{\size}[1]{\left \vert #1 \right \vert}
\newcommand{\ceil}[1]{\left \lceil #1 \right \rceil}

\newcommand{\eps}{\varepsilon}
\newcommand{\E}{\mathbb{E}}
\newcommand{\prob}{\mathbb{P}}
\newcommand{\bin}{\mathrm{Bin}}

\newtheorem{theorem}{Theorem}[section]

\def\nul{\varnothing} 
\def\VEC#1#2#3{#1_{#2},\dots,#1_{#3}}

\def\PE#1#2#3{\prod_{#1=#2}^{#3}}

\def\CH#1#2{\binom{#1}{#2}}
\def\FR#1#2{\frac{#1}{#2}}
\def\FL#1{\left\lfloor{#1}\right\rfloor} 
\def\CL#1{\left\lceil{#1}\right\rceil}

\begin{document}

\title{To Catch a Falling Robber}

\author{
William B. Kinnersley\thanks{Department of Mathematics, University of Rhode
Island, Kingston, RI 02881, billk@uri.edu},\, 
Pawe\l{} Pra\l{}at\thanks{Department of Mathematics, Ryerson University,
Toronto, ON, Canada M5B 2K3, pralat@ryerson.ca. Research supported by NSERC and Ryerson University},\,and 
Douglas B. West\thanks{Departments of Mathematics, Zhejiang Normal University,
Jinhua, China 321004,
and  University of Illinois, Urbana, IL 61801, west@math.uiuc.edu.
Research supported by Recruitment Program of Foreign Experts,
1000 Talent Plan, State Administration of Foreign Experts Affairs, China.}
}

\date{}
\maketitle

\begin{abstract}
We consider a Cops-and-Robber game played on the subsets of an $n$-set.
The robber starts at the full set; the cops start at the empty set.
In each round, each cop moves up one level by gaining an element, and 
the robber moves down one level by discarding an element.  The question is
how many cops are needed to ensure catching the robber when the robber
reaches the middle level.  Alan Hill posed the problem and provided a lower
bound of $2^{n/2}$ for even $n$ and $\CH{n}{\CL{n/2}}2^{-\FL{n/2}}$ (which is asymptotic to $2^{\CL{n/2}}/\sqrt{\pi n/2}$) for odd $n$.
Until now, no nontrivial upper bound was known. In this paper, we prove an upper bound that is within a factor of $O(\ln n)$ of this 
lower bound.
\end{abstract}

\baselineskip=16pt

\section{Introduction}
The game of {\it Cops-and-Robber} is a pursuit game on a graph.
In the classical form, there is one robber and some number of cops.
The players begin by occupying vertices, first the cops and then the robber;
multiple cops may simultaneously occupy the same vertex.
In each subsequent round, each cop and then the robber can move along an
edge to an adjacent vertex.  The cops win if at some point there is a cop
occupying the same vertex as the robber.  The {\it cop number} of a graph $G$,
written $c(G)$, is the least number of cops that can guarantee winning
(all players always know each others' positions).

The game of Cops-and-Robber was independently introduced by
Quilliot~\cite{Qui78} and by Nowakowski and Winkler~\cite{NW83}; both
papers characterized the graphs with cop number $1$.  The cop number as 
a graph invariant was then introduced by Aigner and Fromme~\cite{AF84}.
Analysis of the cop number is the central problem in the study of the game
and often is quite challenging.  The foremost open problem in the field is
Meyniel's conjecture that $c(G) = O(\sqrt{n})$ for every $n$-vertex connected
graph $G$ (first published in~\cite{F}).  
This problem has a relatively long history.  At present we know only that the
cop number is at most $n 2^{-(1+o(1))\sqrt{\log_2 n}}$ (still in $n^{1-o(1)}$)
for any connected graph on $n$ vertices.  This result was obtained
independently by Lu and Peng~\cite{lp}, Scott and Sudakov~\cite{ss}, and
Frieze, Krivelevich, and Loh~\cite{fkl} using probabilistic tools.  For
evidence supporting Meyniel's conjecture, it is natural to check first whether
random graphs provide easy counterexamples.  It is known that Meyniel's
conjecture passes this test for binomial random
graphs~\cite{bpw, bkl, lp2, PW_gnp} and for random $d$-regular
graphs~\cite{PW_gnd}: for connected graphs in these models, the conjecture
holds asymptotically almost surely.
For more background on Cops-and-Robber, see~\cite{BN11}.

We consider a variant of the Cops-and-Robber game on a hypercube, introduced in
the thesis of Alan Hill~\cite{Hil08}.  This variant restricts the initial
positions and the allowed moves.  The $n$-dimensional hypercube is the graph
with vertex set $\{0,1\}^n$ (the set of binary $n$-tuples) in which vertices
are adjacent if and only if they differ in one coordinate.  View the vertices
as subsets of $\{1, \dots,n\}$, and let the {\it $k$th level} consist of the 
$k$-sets -- that is, the vertices whose size as subsets is $k$.  We view 
$\nul$ as the ``bottom'' of the hypercube and $\{1, \dots, n\}$ as the ``top'',
and we say that $S$ lies {\it below} $T$ when $S \subseteq T$.

The robber starts at the full set $\{1,\dots,n\}$; the cops start at the
empty set $\nul$.  On the $k$th round, the cops all move from level $k-1$ to
level $k$, and then the robber moves from level $n+1-k$ to level $n-k$.
If the cops catch the robber, then they do so on round $\CL{n/2}$ at level
$\CL{n/2}$, when they move if $n$ is odd, and by the robber moving onto them if
$n$ is even.

Let $c_n$ denote the minimum number of cops that can guarantee winning the
game.  Hill~\cite{Hil08} provided the lower bound $2^{n/2}$ for even $n$ and
$\CH{n}{\CL{n/2}}2^{-\FL{n/2}}$ for odd $n$; the former bound exceeds the latter
by a factor of $\Theta(\sqrt{n})$. Note that here the cops have in
some sense only one chance to catch the robber, on the middle level.  When
the cops can chase the robber by moving both up and down,
the value is much smaller, with the cop number of the 
$n$-dimensional hypercube graph being $\CL{(n+1)/2}~\cite{MM87}$.

We begin with a proof of Hill's lower bound, since its ideas motivate our
arguments.  (The proof below is essentially Hill's original proof, albeit
presented more concisely.)  We then prove our result: an upper bound within a
factor of $O(\ln n)$ of this lower bound.

\bigskip 

\begin{theorem}[\cite{Hil08}]\label{lower}
$c_n \ge 
\left \{ \begin{array}{ll}
   2^{m},                   &n=2m;\\
   \binom{2m+1}{m+1}2^{-m}, &n=2m+1.                
\end{array} \right.$
\end{theorem}
\begin{proof}
After each move by the robber, some cops may no longer lie below the robber.
Such cops are effectively eliminated from the game.  We call them {\em evaded
cops}; cops not yet evaded are {\em surviving cops}.  

Consider the robber strategy that greedily evades as many cops as possible with
each move.  Deleting an element from the set at the robber's current position
evades all cops whose set contains that element.  On the $k$th round, the
surviving cops sit at sets of size $k$, and the robber has $n-k+1$ choices of
an element to delete.  Since each surviving cop can be evaded in $k$ ways,
the fraction of the surviving cops that the robber can evade on this move is
at least $\FR k{n-k+1}$.

After the first $m$ rounds, where $m=\FL{n/2}$, the fraction of the cops that
survive is at most $\prod_{i=1}^{m} \left ( 1 - \frac{i}{n-i+1} \right )$. 
When $n=2m$, we compute
$$
\PE i1m \left(1-\frac{i}{2m-i+1}\right)
= \PE i1m \frac{2m-2i+1}{2m-i+1}
= \frac{(2m)!}{(2m)!\cdot2^m} = 2^{-m}.
$$
When $n=2m+1$, we compute
$$
\prod_{i=1}^{m} \left(1-\frac{i}{2m-i+2} \right)
= \prod_{i=1}^{m} \frac{2m-2i+2}{2m-i+2}
= \frac{2^m m! (m+1)!}{(2m+1)!}
= 2^m \Big/ \binom{2m+1}{m+1}.
$$
For the cops to catch the robber, at least one surviving cop must remain after 
$m$ moves; this requires at least $2^m$ total cops when $n=2m$ and at least
$\binom{2m+1}{m+1}2^{-m}$ when $n=2m+1$.
\end{proof}

A similarly randomized strategy for the cops should produce a good upper bound.
However, it is difficult to control the deviations from expected behavior
over all the cops together.  Our strategy will group the play of the game into
phases that enable us to give essentially the same bound on undesirable
deviations in each phase.

\section{The Upper Bound}

If there are enough cops to cover the entire middle level, then the robber
cannot sneak through.  The size of the middle level is asymptotic to 
$2^n/\sqrt{\pi n/2}$.  This trivial upper bound is roughly the square of the
lower bound in Theorem~\ref{lower}.  When $n$ is odd, a slight improvement
follows by observing that one only needs to block each $(n+1)/2$-set by
reaching some $(n-1)/2$-set under it.  More substantial improvements use the
fact that as the robber starts to move, the family of sets needing to be
protected shrinks.

Our upper bound on $c_n$ is $O(\ln n)$ times the lower bound in
Theorem~\ref{lower}.  We use a randomized strategy for the cops; it may or may
not succeed in capturing the robber.  However, with sufficiently many cops, the
strategy succeeds {\em asymptotically almost surely} (or {\em a.a.s.}), that
is, with probability tending to 1 as $n$ tends to infinity.  Consequently,
some deterministic strategy for the cops (in response to the moves by
the robber) wins the game.

To analyze our cop strategy, we need a version of the well-known Chernoff Bound: 

\begin{theorem}[\cite{JLR}]\label{thm:Chernoff}
Let $X$ be a random variable expressed as the sum $\sum_{i=1}^n X_i$ of
independent indicator random variables $\VEC X1n$, where $X_i$ is a Bernoulli
random variable with expectation $p_i$ (the expectations need not be equal).
For $0 \le \eps \le 1$,
\begin{eqnarray*}
\prob\left [X \le (1-\eps)\E[X]\right ] &\le& \exp \left( - \frac{\eps^2 \E[X]}{2} \right).
\end{eqnarray*}
\end{theorem}

\nobreak
We are now ready to prove our result.

\bigskip 

\begin{theorem}\label{upper}
$c_n =
\left \{ \begin{array}{ll}
   O(2^{m}\ln n ),                   &n=2m;\\
   O(2^{-m}\binom{2m+1}{m+1}\ln n ), &n=2m+1.                
\end{array} \right.$
\end{theorem}
\begin{proof}
We consider the case $n=2m$ first, returning later to the case $n=2m+1$.

We will specify the number of cops later.  All the cops begin at $\nul$.  Let
$R$ be the current set occupied by the robber.  On his $k$th turn, for
$1 \le k \le m$, each surviving cop at set $C$ chooses the next element for his
set uniformly at random from among $R-C$.  We claim that, regardless of how the
robber moves, this cop strategy succeeds a.a.s.

To facilitate analysis of the cops' strategy, we introduce some notation and 
terminology.  Consider an instance of the game.  We say that this 
instance satisfies {\em property $P(t,a)$} if, after $t$ rounds, every $m$-set
below the robber also has at least $a$ cops at or below it.  Intuitively, the
$m$-sets below the robber are the places where the robber can potentially be
captured; property $P(t,a)$ means that each of them can be reached by at least
$a$ cops.

To show that the cop strategy a.a.s.\ captures the robber, we will show that,
no matter how the robber plays, a.a.s.\ property $P(t_i,a_i)$ holds for
specific choices of $t_i$ and $a_i$.  Let $r=\ceil{\log_2\log_2n}$, and for
$i \in \{0, \dots, r\}$ let $s_i = 2^{r-i}$ and $t_i = m - s_i$.  
Furthermore, let 
$$a_i = 1600\left ( \prod_{j=1}^{i} (1-\eps_j)\right )2^{s_i}\ln n, \quad \quad \text{ where } \eps_j = \sqrt{s_j/2^{s_j}}.$$
In particular, $a_0 = 1600\cdot 2^{2^r}\ln n$.  Note that always
\begin{align*}
\prod_{j=1}^i (1-\eps_j) &\ge \prod_{j=1}^r (1-\eps_j)\\ 
                         &\ge \exp\left (-2\sum_{j=1}^r\eps_j\right )\\
                         &\ge \exp \left (-2 \left (\sqrt{2^0/2^{2^0}} + \sqrt{2^1/2^{2^1}} + \sqrt{2^2/2^{2^2}}(1 + 1/2 + 1/4 + \dots)\right )\right )\\
												 &= \exp(-2\sqrt{2}-2) > 1/200,
\end{align*}
and hence $a_i \ge 8\cdot 2^{s_i}\ln n$.  (Above, the second inequality uses
the fact that $1-x \ge \exp(-2x)$ whenever $0 \le x \le 1/\sqrt{2}$, while the
third inequality uses the observation that $\eps_{j-1} \le \eps_j/2$ for
$0 \le j \le r-2$.) 
 
We play the game with $\ceil{3200 \cdot 2^{m} \ln n}$ cops.  We claim that
a.a.s.\ property $P(t_i,a_i)$ holds for all $i$ in $\{0, \dots, r\}$.  We also
claim that a.a.s.\ property $P(m,1)$ holds.  This ensures that in the final
round the cops can cover all vertices where the robber can move; hence they
win.

We break the game into $r+2$ {\em phases}.  Phase 0 consists of rounds $1$
through $t_0$.  For $i \in \{1, \dots, r\}$, Phase $i$ consists of rounds
$t_{i-1}+1$ through $t_i$.  Phase $r+1$ consists of the single round $t_r+1$.
Our analysis is inductive.  For Phase 0, we show that a.a.s. property
$P(t_0,a_0)$ holds.  When considering Phase $i$ for $1 \le i \le r$, we assume
that property $P(t_{i-1},a_{i-1})$ holds and show that a.a.s.\ property
$P(t_i,a_i)$ also holds.  Finally, for Phase $r+1$, we assume that property
$P(t_r,a_r)$ holds and show that a.a.s.\ the cops capture the robber.

We begin with Phase 0.  We claim that property $P(t_0,a_0)$ holds with
probability at least $1-1/n$, no matter how the robber moves.  Fix a sequence
of moves for the robber in the first $t_0$ rounds of the game, and fix a set
$S$ with $\size{S} = m$ that remains below the robber.  A particular cop
remains below $S$ if and only if his position contains only elements of $S$.
In round $i$, each cop below $S$ has already added $i-1$ such elements, and
$m-i+1$ others remain.  Since each surviving cop chooses a new element
uniformly from $2m-2i+2$ possibilities, the probability that a cop below $S$
remains below $S$ is $\frac{m-i+1}{2m-2i+2}$, which equals $1/2$.  Thus, a
given cop remains below $S$ after the first $t_0$ rounds with probability
$2^{-t_0}$.

Consequently, the number of cops remaining below $S$ after $t_0$ rounds is a
random variable $X$ with the binomial distribution
$\bin(\ceil{3200\cdot 2^{m}\ln n}, 2^{-t_0})$.  Recalling that $t_0 = m-s_0$
and that $s_0 = 2^r \ge \log_2 n$, we have 
$$\E[X] \ge 3200\cdot 2^{m}\ln n \cdot 2^{-t_0} = 3200 \cdot 2^{s_0} \ln n = 3200 \cdot 2^{2^r} \ln n = 2a_0.$$
The Chernoff Bound now yields 
$$\prob(X \le a_0) \le \prob\left (X \le \frac{\E[X]}{2}\right ) \le \exp\left (-\frac{(1/2)^2\E[X]}{2}\right ) < \exp(-3n\ln n).$$

Thus, the probability that fewer than $a_0$ cops remain below $S$ is less than
$\exp(-3n \ln n)$.  The number of such sets $S$ below the robber is less than
$2^n$, which is less than $\exp(n \ln n)$.  By the Union Bound, the probability
that some $m$-set below the robber has fewer than $a_0$ cops below it is thus
less than $\exp(-2n\ln n)$.  That is, for one sequence of moves by the robber,
property $P(t_0,a_0)$ fails to hold with probability at most $\exp(-2n\ln n)$.
The number of possible move sequences by the robber in Phase 0 is less than
$n^{t_0}$, which in turn is less than $\exp(n \ln n)$.  Again using the Union
Bound, the probability that some robber strategy causes property $P(t_0,a_0)$
to fail is less than $\exp(-n \ln n)$.  Thus property $P(t_0,a_0)$ holds with
probability more than $1-\exp(-n\ln n)$, which is more than $1-1/n$.

Next consider Phase $i$ with $1\le i\le r$, consisting of rounds $t_{i-1}+1$
through $t_i$.  Under the assumption that property $P(t_{i-1},a_{i-1})$ holds,
we claim that property $P(t_i,a_i)$ also holds with probability at least
$1-1/n$.  The argument is similar to that for Phase 0.  Fix a sequence of moves
for the robber in rounds $t_{i-1}+1$ through $t_i$, and fix an $m$-set $S$ that
remains below the robber after round $t_i$.  Again a cop below $S$ on a given
round remains below $S$ after that round with probability $1/2$.

By assumption, at least $a_{i-1}$ cops sat below $S$ at the beginning of Phase
$i$; the number of cops remaining below $S$ at the end of Phase $i$ is thus
bounded below by the random variable $X$ with binomial distribution
$\bin(\ceil{a_{i-1}}, 2^{-(t_i-t_{t-1})})$.  Hence
$$\E[X] \ge a_{i-1}2^{t_{i-1}-t_i} = a_{i-1}2^{s_i-s_{i-1}} = a_i \cdot \frac{a_{i-1}\cdot 2^{-s_{i-1}}}{a_i\cdot 2^{-s_i}} = \frac{a_i}{1-\eps_i}.$$
This time, the Chernoff Bound yields
\begin{align*}
\prob(X \le a_i) &\le \prob(X \le (1-\eps_i)\E[X]) \le \exp\left(-\frac{\eps_i^2 \cdot \E[X]}{2}\right ) \le \exp\left(-\frac{\eps_i^2 \cdot a_i}{2}\right )\\
              &\le \exp\left (-\frac{\eps_i^2 \cdot 8 \cdot 2^{s_i} \ln n}{2}\right ) = \exp(-4s_i \ln n).
\end{align*}
At the start of Phase $i$, the robber occupies level $n-t_{i-1}$.  At this time, 
the number of $m$-sets that lie below the robber is $\binom{n-t_{i-1}}{m}$.  
This simplifies to $\binom{m+s_{i-1}}{m}$, which is at most $n^{s_{i-1}}$; 
since $s_{i-1} = 2s_i$, this is at most $\exp(2s_i \ln n)$.  Likewise, the  
number of move sequences available to the robber during Phase $i$ is at most 
$n^{s_{i-1}-s_i}$, which simplifies
to $\exp(s_i\ln n)$.  Applying the Union Bound twice,
as in Phase $0$, we see that property $P(t_i,a_i)$ fails with probability at
most $\exp(-s_i\ln n)$.  Hence $P(t_i,a_i)$ holds with probability at least
$1-\exp(-s_i\ln n)$, which is at least $1-1/n$.

Finally, we show that if $P(t_r,a_r)$ holds, then $P(m,1)$ holds with
probability at least $1-1/n$.
Recall that $t_r = m-1$ and that $a_r \ge 16 \ln n$.  Each cop chooses from two
possible moves, each leading to an $m$-set.  The number of cops that remain
below an $m$-set $S$ is bounded from below by the random variable $X$ with
distribution $\bin(\ceil{a_{r}}, 1/2))$.  Now 
$$\prob(X = 0) = 2^{-\ceil{a_r}} \le 2^{-16 \ln n} \le \frac{1}{n^2},$$
so the probability that no cop reaches $S$ is at most $1/n^2$.  There are $m+1$
choices for $S$; by the Union Bound, $P(m,1)$ fails with probability less than
$1/n$.  Hence $P(m,1)$ holds with probability at least $1-1/n$, as claimed.

To complete the proof, we now consider the full game.  We want to show that
a.a.s. $P(m,1)$ holds.  The probability that $P(m,1)$ holds is bounded below by
the probability that $P(t_0,a_0), \ldots, P(t_r,a_r)$, and $P(m,1)$ all hold.
We have shown that $P(t_0,a_0)$ fails with probability at most $1/n$, that
$P(t_i,a_i)$ for $1\le i\le r$ fails with probability at most $1/n$ when
$P(t_{i-1},a_{i-1})$ holds, and that $P(m,1)$ fails with probability at most
$1/n$ when $P(t_r,a_r)$ holds.  By the Union Bound, the probability that some
property in this list fails is bounded above by $(r+2)/n$, which is at most
$2 \log_2 \log_2 n / n$ when $n$ is sufficiently large.  Thus the conjunction
of these properties (and in particular, property $P(m,1)$) holds with
probability at least $1-2\log_2\log_2 n/n$.  This completes the proof for the
case $n=2m$.

When $n=2m+1$, we define property $P(t,a)$ to mean that after $t$ rounds, at
least $a$ cops sit below each $(m+1)$-set that is below the robber.  It now
suffices to prove that $P(m,1)$ holds a.a.s., since any cop that remains below
the robber at the beginning of round $m+1$ can capture him.  The details of the
argument are nearly identical to the previous case, and we omit them.
\end{proof}

We remark that the cops can play more efficiently by using an appropriate
deterministic strategy in round $m$.  This does not improve the asymptotics of
our bound, but it does improve the leading constant.

\end{document}